\documentclass[12pt, reqno]{amsart}
\usepackage{amsmath, amsthm, amscd, amsfonts, amssymb, graphicx, color}
\usepackage[bookmarksnumbered, colorlinks, plainpages]{hyperref}

\newtheorem{theorem}{Theorem}[section]
\newtheorem{lemma}[theorem]{Lemma}

\newtheorem{corollary}[theorem]{Corollary}
\theoremstyle{definition}
\newtheorem{definition}[theorem]{Definition}

\theoremstyle{remark}
\newtheorem{remark}[theorem]{Remark}
\numberwithin{equation}{section}

\begin{document}
\setcounter{page}{1}

\title[construction of subfactor from planar algebra]{A construction of subfactor by planar structure}

\author[WungHun Ri]{WungHun Ri ,  GwangHo Jong $^1$}

\address{$^{1)}$ Department of Mathematics, Kim Il Sung University, Pyongyang, D. P. R. Korea.}
\email{\textcolor[rgb]{0.00,0.00,0.84}{leewunghun@yahoo.com}}



\date{Submtted: 28 Oct 2012 ; Revised: 20 Feb 2013}

\begin{abstract}
We present more planar algebraic construction of subfactors than those of Guionnet-Jones-Shlyakhtenko-Walker and Kodiyalam-Sunder which start from a subfactor planar algebra and give in a direct way a subfactor of the same standard invariant with the planar algebra. Our construction is based on using the ordinary concepts in planar algebras such as involution, inclusion and conditional expectation mappings as it is.
\end{abstract} \maketitle

\section{Introduction}

\noindent This paper was motivated by a joint work of Jones and his colleagues ([6]).
Prier to it, starting from any given subfactor planar algebra $P=(P_{n})_{n=0_{\pm},1,2,\cdots}$, a construction of a subfactor whose standard invariant is precisely the planar algebra was given by the Jones and his colleagues in [1]. It gave a diagrammatic reproof of the remarkable result of Popa in [12]. [6] was proposed as much more simplified approach to the main result in [1]. Their construction is based on giving the structure of Hilbert algebra to their graded vector space $Gr_k(P)=\bigoplus_{n=0}^\infty P_{n+k}$ . Incidentally in [6], explaining the $\ast$-structure on the direct sum componentwise,  they described the $\ast$-operation on $P_{n,k}:=P_{n+k}$  as being just the involution coming from the subfactor planar algebra ([6], definition 3.1).
This explanation seems a little loose; it is not difficult to see that the original (or ordinary) involution in the subfactor planar algebra is not consistent with their pictorial convention in [6] about the elements in $P_{n,k}$ and moreover with their other algebraic structures such as the graded product.
In fact, what they meant was the one given in [1] which is precisely different from the ordinary involution of the planar algebra.
Nevertheless, together with the tangles for Jones projections, inclusions and conditional expectations, the involution operation is one of the basic ingredients which not only determine the subfactor planar algebras, but also are most meaningful in connection with subfactor theory ([4], [5], [7], [11]).
On the other hand, even more, the ordinary concepts such as inclusion and conditional expectation are also meaningless within their construction.
Namely, within $\{P_{n,k}(=P_{n+k})\}$  the germinal ingredients for the involutions, inclusions and conditional expectations in the out coming subfactor are not the same with the ordinary ones for the planar algebras. (for more detailed discussion, see [8] of Kodiyalam and Sunder which gives substantially the same construction with [6])
In this connection we are interested in finding a possibility of another construction in the same spirit as [1], [6] and [8], but by using the ordinary algebraic concepts in the given subfactor planar algebra, especially the involution, inclusion and conditional expectation intact. If one can find such a construction, it would be called more planar than above mentioned.

Unfortunately, however we choose the distinguished interval delicately, any attempt to make such a construction upon their frame could not be succeeded. In other words, based on their irect sum $\bigoplus_{n=0}^\infty P_{n,k}\ (=\bigoplus_{n=0}^\infty P_{n+k}=\bigoplus_{n=k}^\infty P_{n})$, it is impossible to reconcile their graded product with those standard and ordinary concepts of subfactor planar algebra such as involution, inclusion and the others already existed.
Recently we noticed that an alteration of explanation of summands in their direct sum $\bigoplus_{n=0}^\infty P_{n,k}$ gives such a possibility, i.e. a way of constructing of subfactors by using the ordinary algebraic concepts given in the subfactor planar algebra - the involution, inclusions and conditional expectations as it is.
Our approach follows the line of [6] mainly, but needs slight modifications in some details and gives a new construction of a subfactor (more exactly, a tower of subfactors) whose standard invariant is precisely the given subfactor planar algebra as well.
It is seems that there would be no equivalence between our construction(as a model in the sense of [8]) and above mentioned ones ([1], [6], [8]) which are all equivalent.
Moreover our approach gives for any given subfactor planar algebra, an infinite family of towers of subfactors with the same standard invariant, but seemingly of quite different classes.

\section{From planar algebras to Hilbert algebras}
Let us begin with a given subfactor planar algebra $P = (P_{n})_{n=0_{\pm},1,2,\cdots\ }$.
By definition every $P_{n}$ is a finite dimensional $C^*$-algebra.
On the other hand they are also an inner product spaces by a non-degenerate sesquilinear form $(a,b)\mapsto Tr(b^\ast a)$  given by the following diagram(``trace tangle"):
\begin{figure}[ht]
\centering
\includegraphics[width=0.35\textwidth]{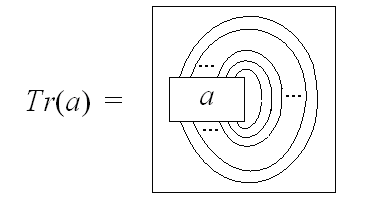}
\end{figure}
\begin{definition} Let $k=0,1,2,\cdots$. On
$$
H_k(P):=P_{k}\oplus P_{k+2}\oplus P_{k+4}\oplus P_{k+6}\oplus \cdots = \bigoplus_{n=0}^\infty 
P_{k+2n}
$$         
as a direct sum of inner product spaces, an involution is given from the $C^*$-structure of $P$ componentwise.
\end{definition}
Moreover, with notation $P_{n,k}:=P_{2n+k}$ we use the expression$H_k(P)=\bigoplus_{n=0}^\infty P_{n,k}$ in a manner analogous to [6].
But it should be emphasized the essential difference, in the meaning of notation $P_{n,k}$, between ours and the one ($P_{n,k}=P_{n+k}$) in [6].
For a while, by $Q_k(\ ,\ )$ denote the inner product in $H_k(P)$.

Let us picture the tangle representing an element $a\in P_{n,k}$ in the following way: 
\begin{figure}[h]
\centering
\includegraphics[width=0.4\textwidth]{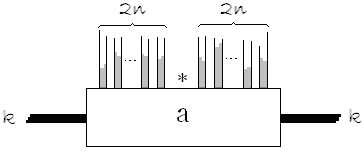}
\caption{}
\end{figure}
\\
Here the distinguished interval is placed in the center of the upper face (the starlike mark) and respectively $k$ strings (tied into thick line) run out  from the both sides of the box. The rest of the strings, i.e., $4n$ strings are stretched up from the upper face and, as we can see, there is shading with white and black alternately in this upper region.
The outer box of the tangle is suppressed as well. It does not lead to any confusion.
As the region touching the distinguished interval is white, it is clear that the both regions touching the upper corners are also white. Therefore the upper part of the tangle is consisted with $2n$ black regions bordered by $4n$ strings. The essential point is that we are mainly concerned with the situation where two adjacent strings with a common black region behave like a couple. In such a situation it is possible to consider each black region (called black band or simply band) like a line. For this reason we redraw the elements in $P_{n,k}$ like the following one which have no difference with [6] in appearance:
\begin{figure}[hb]
\centering
\includegraphics[width=0.25\textwidth]{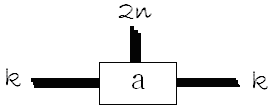}
\caption{}
\end{figure}
\\
But according to our convention the thick line on upper face means a bundle of $2n$ black bands while each one in both sides represents respectively a bundle of $k$ strings.
To indicate the distinguished interval, it is sufficient to stress that the bottom of the box is opposite to the distinguished one.

\begin{remark} 1) On our way of describing the result of involution to a tangle representing an element in $P_{n,k}$ is simply horizontal reflection of the picture with replacing all the elements in the boxes by involution of them respectively.

2) For $a,b\in P_{n,k}$, $Q_k(a,b)=Tr(b^*a)$ is represented in our convention as following:
\begin{figure}[hb]
\centering
\includegraphics[width=0.28\textwidth]{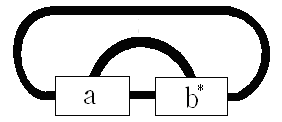}
\caption{}
\end{figure}
\end{remark}
\begin{definition} For $a\in P_{m,k}$ and $b\in P_{n,k}$, their product $a\circ b$ is given by the following diagram:
\begin{figure}[ht]
\centering
\includegraphics[width=0.38\textwidth]{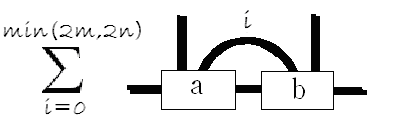}
\end{figure}
\\
It should be noticed that in this figure the number $i$ over the arc means the number of bands, not strings.
Through linear combinations, on $H_k(P)=\bigoplus_{n=0}^\infty P_{n,k}$ a multiplication $(a,b)\mapsto a\circ b$ is introduced.
\end{definition}

\begin{lemma}\label{main}
With the multiplication defined above, $H_k(P)$ is an associative unital $\ast$-algebra.
\end{lemma}

\begin{proof} The unit of $C^\ast$-algebra $P_k$ is represented by the following trivial diagram as an element in $P_{0,k}\subset H_k(P)$:
\begin{figure}[ht]
\centering
\includegraphics[width=0.2\textwidth]{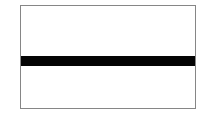}
\end{figure}
\\
It is clear that this tangle gives the unit in $H_k(P)$.

To verify the equality $(a\circ b)^\ast =a^\ast \circ b^*$, it is sufficient to see the below pictorial equality which is clear from 1) in the remark 2.2.
\begin{figure}[ht]
\centering
\includegraphics[width=0.62\textwidth]{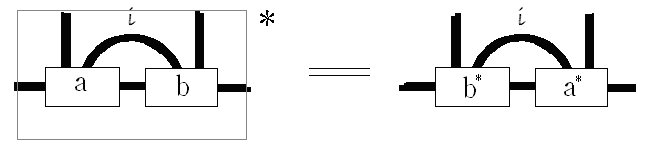}
\end{figure}
\\
We omit the verification of associative rule since it would be analogous to one in [6].
\end{proof}
Due to above consideration, every $H_k(P)$ has both a pre-Hilbert space structure and a $\ast$-algebraic one.

\begin{remark} Let us recall here that a pre-Hilbert space $\mathfrak{A}$ which is at once a $*$-algebra is called a Hilbert algebra if it satisfies the following conditions (1)-(4):

(1) $\langle a,b\rangle=\langle b^*,a^*\rangle,\ a,b\in \mathfrak{A}$

(2) $\langle ab,c\rangle=\langle b,a^*c\rangle,\ a,b,c\in \mathfrak{A}$

(3) For every $a\in \mathfrak{A}$, the left multiplication $\mathfrak{A}\ni a\mapsto ab\in \mathfrak{A}$ gives a bounded operator.

(4) The vector subspace spanned by $\{ab\rvert\ a,b\in \mathfrak{A}\}$ is dense in $\mathfrak{A}$.
\end{remark}
It is well known that Hilbert algebras give von Neumann algebras associated with them. The main purpose of this section is to show that $H_k(P)$ is a Hilbert algebra.
Prier to that let us consider the relation between $H_k(P)$ for various $k$.
\begin{definition} Let $k\leq l$. If we regard the ordinary inclusions in the given planar algebra
$$P_{2n+k}\subset P_{2n+l},\ n=0,1,\cdots,$$
as being $P_{n,k}\subset P_{n,l}$, $n=0,1,\cdots$, i.e., the inclusions between the components for $H_k(P)$ and $H_l(P)$, then the corresponding tangles look like the following diagram(Fig.4):
\begin{figure}[ht]
\centering
\includegraphics[width=0.36\textwidth]{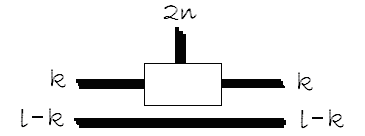}
\caption{}
\end{figure}
\\
Therefore a natural inclusion of (primarily) vector spaces, $I_k^l: H_k(P)\rightarrow H_l(P)$ is defined componentwise.
\end{definition}
For the sake of convenience in notation, from now on we suppose that the inner product in $H_k(P)$ is normalized by $\langle a,b\rangle_k:=\delta ^{-k}Q_k(a,b)$. Here $\delta>0$ denote the modulus of the planar algebra $P$ and will be fixed throughout this paper.
\begin{lemma}
inclusion $\mathcal{I}_k^l: H_k(P)\rightarrow H_l(P)$ is a $*$-algebra isomorphism preserving the inner products.
\end{lemma}

\begin{proof} All needed are clear by associating Fig.4 with remarks 1.2 and definition 1.3.
For instance, in view of Fig.3 for $a,b\in P_{n,k}$, we can see that $Q_l(\mathcal{I}_k^l(a),\mathcal{I}_k^l(b))$ has $l-k$ extra strings than $Q_k(a,b)$, therefore we obtain the following equality:
$$ \langle \mathcal{I}_k^l(a),\mathcal{I}_k^l(b)\rangle_l=\delta ^{-l}Q_l(\mathcal{I}_k^l(a),\mathcal{I}_k^l(b))=\delta ^{-k}Q_k(a,b)=\langle a,b\rangle_k $$
\end{proof}
Due to the above lemma we can omit the number $k$ from $\langle\ ,\ \rangle_k$, the notation of the inner product in $H_k(P)$.

\begin{theorem} For each $k$, $(H_k(p),\langle\ ,\ \rangle,\circ,\ast)$ is a Hilbert algebra.
\end{theorem}
\begin{proof} From remark 1.2, the condition (1) in remark 1.5 is clear and the existence of the unit in $(H_k(p)$ guarantees the condition (4).

Let us verify the condition (2). Let $a\in P_{m,k}$, $b\in P_{n,k}$ and $c\in P_{l,k}$ be given arbitrarily. It is sufficient to prove that the equality $\langle a\circ b,c\rangle=\langle b,a^*\circ c\rangle$, or $Q_k(a\circ b,c)=Q_k(b,a^* \circ c)$ equivalent to it.
Clearly
\begin{align*}
& a\circ b\in P_{r,k}\oplus P_{r+1,k}\oplus P_{r+2,k}\cdots \oplus P_{m+n,k},\\
& \qquad \qquad \qquad \qquad r:=m+n-2 min(m,n)=\lvert m-n \rvert
\end{align*}
from the definition 1.3. Besides, it is easy to see that, in cases of $l<\lvert m-n \rvert$ or $m+n<l$ for  $c\in P_{l,k}$, $b$ is orthogonal to $a^*\circ c$. Indeed if $l<\lvert m-n \rvert$, then $m+l<n$ or $n<m-l$, and if $m+n<l$, then $n<l-m$. Therefore our consideration reduces to the cases where $\lvert m-n \rvert\leq l\leq m+n$ is satisfied.

Since $c\in P_{l,k}$ is orthogonal to $a\circ b$ except the component in $P_{l,k}, Q_k(a\circ b,c)$ can be expressed as follows:
\begin{figure}[ht]
\centering
\includegraphics[width=0.4\textwidth]{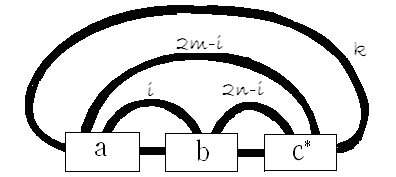}
\end{figure}
\\
This can be altered by planar isotopy into the following diagram.
\\
\begin{figure}[ht]
\centering
\includegraphics[width=0.4\textwidth]{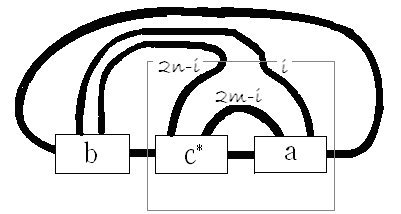}
\end{figure}
\\
After performance of involution to the tangle enclosed by the light line, by putting $j=2m-i$ above picture is transformed to the following one which obviously represents $Q_k(b,a^*\circ c)$:
\begin{figure}[ht]
\centering
\includegraphics[width=0.42\textwidth]{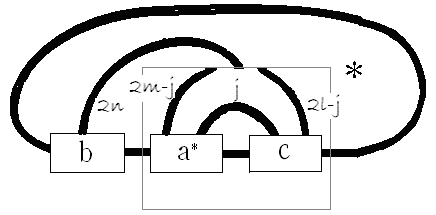}
\end{figure}

Now we turn to the verification of the condition (3), i.e. the proof of boundedness of the operator
$$L_a:\ H_k(P)\ni b\mapsto a\circ b\in H_k(P)$$
for any given $a \in H_k(P)$.

We may assume, without loss of generality, $a \in H_k(P)$ for some $n$. $L_a$ is decomposed into a sum of the following $2n+1$ operators $L_a^i,\ i=0,1,\cdots,2n$.
$L_a^i$ vanishes on $P_{l,k}$ in case $2l<i$, and gives for $x\in P_{l,k}$ in case $2l\geq i$, the element represented by this diagram:
\begin{figure}[ht]
\centering
\includegraphics[width=0.27\textwidth]{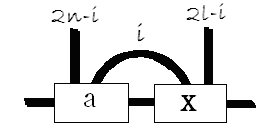}
\end{figure}

Since the boundedness of $L_a$ is equivalent to those of all $L_a^i$, it is sufficient to prove that every $L_a^i$ is bounded for any fixed $i=0,1,\cdots,2n$. What is clear but peculiar is that $L_a^i$ maps every component in $H_k(P)=\bigoplus_{n=0}^\infty P_{n,k}$ into also a component of the direct sum.
This means that for every $l,\ 2l\geq i$, the restriction of $L_a^i$ to $P_{l,k}$ may be denoted by $L_a^i:P_{l,k}\rightarrow P_{l+n-i,k}$.
Therefore it is clear that if all the restrictions $L_a^i:P_{l,k}\rightarrow P_{l+n-i,k}$ are bounded uniformly over all $l$, then $L_a^i$ would be also bounded.

Now let us estimate the norms of $L_a^i:\ P_{l,k}\rightarrow P_{l+n-i,k}$ for every $l,\ 2l\geq i$.
For $x\in P_{l,k}$, $Q_k(L_a^i x,L_a^i x)=Q_k(a\circ x,a\circ x)$ can be seen as the tangle in Fig.5.

By lemma 2.7 even if the consideration is referred to any $H_m(P)$ such as $H_k(P)\subset H_m(P)$, it does no matter with the boundedness of $L_a^i$. Therefore by a suitable embedding if need be, there is no loss of generality in assuming $k=2k'$ for some integer $k'\geq 0$.
Now we regard even the $k=2k'$ stings running from the both sides in Fig.5 as $k'$ black bands similarly to those from upper faces of the boxes, so that all the thick lines mean the bundles of black bands.
\begin{figure}[ht]
\centering
\includegraphics[width=0.45\textwidth]{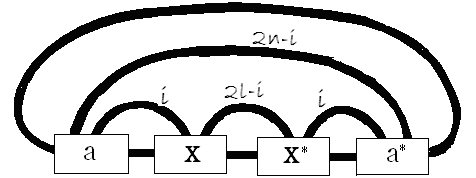}
\caption{}
\end{figure}
\\
As a result $k'$ bands touch both sides of every boxes in the figure respectively, so the number of bands for boxes labeled by $a$ and $x$ is $2(n+k')$ and $2(l+k')$ respectively.

Let us see first the case $i\leq n$.
Under above convention we perform spherical isotopy to Fig.5 transforming it to the left one in Fig.6 where the same number of bands($l+n-i+k'$) touch the bottom and top of each of the $4$ imaginary sections divided by the light horizontal lines.
\begin{figure}[ht]
\centering
\includegraphics[width=0.57\textwidth]{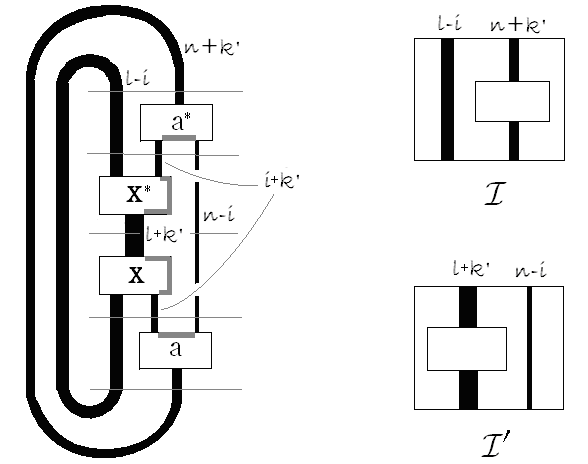}
\caption{}
\end{figure}
\\
What should be careful is that about the positions of distinguished intervals. For each box the thick portion in the border indicates the opposite interval to the distinguished one. In the right in Fig.6, two tangles which describe the ordinary inclusions into $P_{2m_l+2k'}$, $m_l:=l+n-i$, in the standard way([5]). In other words, for each of the boxes in $\mathcal{I}$ and $\mathcal{I}'$ the left side indicates the distinguished interval. Note that the numbers on both tangles denote the number of bands, not strings.

Putting $a_l:=\mathcal{IR}^{n-i}(a),\ x':=\mathcal{I}'(x)$ by using the two embeddings and a suitable rotation([5])  $\mathcal{R}$, the left tangle in the figure means $Tr_{2m_l+2k'}(a_l^*x'^*x'a_l)$ involving the ordinary multiplications in the planar algebra $P$. Again, in view of identification $P_{2m_l+2k'}=P_{m_l,k}$, using positivity of the traces and $C^*$-property we obtain the following estimation:
\begin{align*}
&\delta^k\langle a\circ x,a\circ x \rangle=Q_k(a\circ x,a\circ x)\nonumber\\
&=Tr_{2 m_l+2 k'}(a_l^*x'^*x'a_l)=Tr_{2 m_l+2 k'}(x'a_l a_l^*x'^*)\\
&\leq \lVert a_l a_l^* \rVert Tr_{2 m_l+2 k'}(x'^*x')=\lVert a_l \rVert ^2 Tr_{2 m_l+2 k'}(x'^*x')\\
&=\lVert a_l \rVert ^2 Q_k(x'^*x')=\lVert a_l \rVert^2\delta^k\langle x'^*,x'\rangle
\end{align*}
Here $\lVert \cdot \rVert$ denotes the $C^*$-norm of $P_{2m_l+2k'}$.
By uniqueness of $C^*$-norm or, as equivalently, conservation of norm under $\ast $-isomorphism for $C^*$-algebras, for each $l$, $\mathcal{I}:P_{2n+k}\rightarrow P_{2m_l+k}$ is an isometry. Therefore $\lVert a_l \rVert=\lVert \mathcal{R}^{n-i}(a)\rVert$ and after all $\lVert a_l \rVert$ are independent of $l$.
On other hand, since $\mathcal{I}'$ is the embedding in lemma 2.7, and from above inequality we obtain the following estimations for all $l$:
$$
\langle L_a^i x,L_a^i x\rangle\leq C_a^i\langle x,x \rangle,\ x \in P_{l,k}\ \ (\ C_a^i:=\lVert \mathcal{R}^{n-i}(a) \rVert \ )
$$     
In case $i\leq n$, it turned out that for all $l$, $L_a^i:\ P_{l,k}\rightarrow P_{l+n-i,k}$ are uniformly bounded.

Next let us examine the case $i>n$. If we add $n$ closed bands around the tangle in Fig.5, then the following one equivalent to $\delta^{2n}Q_k(L_a^i x,L_a^i x)$, is obtained:
\begin{figure}[ht]
\centering
\includegraphics[width=0.5\textwidth]{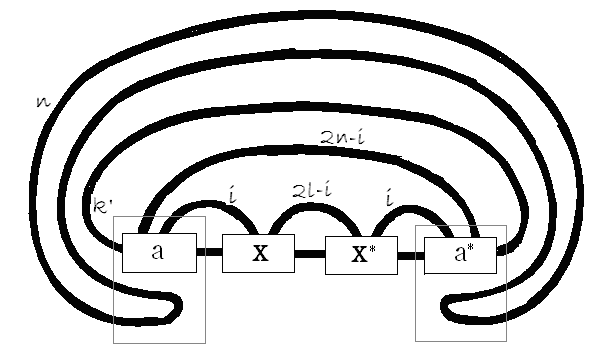}
\end{figure}
\\
Now consider $a \in P_{n,k}(=P_{2(n+k')})$ to be embedded in $P_{2(2n+k')}$ through the standard $ \ast $-inclusion $P_{2(n+k')} \hookrightarrow P_{2(2n+k')}$ described by the left one in the following figure. $\tilde a $ denotes the result of applying to $a \in P_{2(2n+k')}$ the relevant rotation by $n$ times. Then reconsideration of $P_{2(2n+k')}$ as $P_{2n,k}=P_{2n,2k'}$ gives a simplification of above figure to the right one in the following:
\begin{figure}[ht]
\centering
\includegraphics[width=0.6\textwidth]{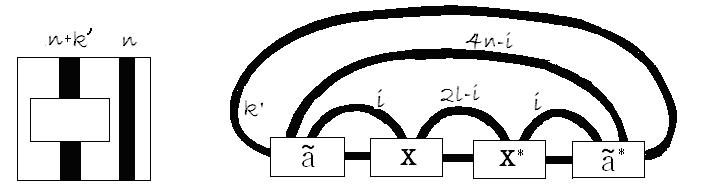}
\end{figure}
\\
After all $ \delta^{2n}Q_k(L_a^i x,L_a^i x)=Q_k(L_ {\tilde a}^i x,L_{\tilde a}^i x)$ and since $i<2n$ for $\tilde a \in P_{2n,k}$, the required norm estimation returns to the already considered case.
\end{proof}

\begin{remark} 1) For every $k=0,1,2,\cdots$, denote the left von Neumann algebra of $H_(P)$ as a Hilbert algebra by $M_k$. Therefore for $x,y \in M_k$, the product $xy$ in $M_k$ coincides with $x\circ y$ in case $x,y \in H_k(P)$.

2) Since $H_k(P)$ is unital, $M_k$ is a finite von Neumann algebra. Moreover the canonical trace on $M_k$ is given by this:
$$ tr:M_k \rightarrow \mathbb{C},\ tr(x)=\langle xI,I \rangle $$  
Here $I$ denotes the unit of $H_k(P)$ and $xI$ means the action of $x$(as an operator) to $I$. Notice that especially in case $x \in P_{n,k}(\subset H_k(P))$, except only the case where $n=0$, we always have $tr(x)=0$.

3) As usual we denote the completion of $H_k(p)$ by $L^2(M_k,tr)$ or $L^2(M_k)$.
\end{remark}
Of course, by the embeddings in lemma 2.7 we have the inclusion tower of von Neumann algebras
$M_0 \subset M_1 \subset M_2 \subset M_3 \subset \cdots$.

\section{the von Neumann algebras constructed from the planar algebras}
Here we prove that, for every $k=0,1,2,\cdots$, the von Neumann algebra $M_k$ is a factor of type $II_1$.

Let us begin with any fixed $k$. As to tangles, unless we explicitly say otherwise, the lines running upward will denote the black bands as before.
Let us consider two elements given by the following diagrams (For $W$, a detailed explanation by spreading the lines to black regions is also given):
\begin{figure}[ht]
\centering
\includegraphics[width=0.67\textwidth]{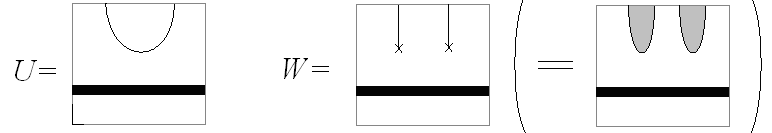}
\end{figure}
\begin{remark} 1) Let $\mathfrak{A}$ denote the subalgebra generated by $U$ in $H_k(P)$. Then multiplication by $U$ gives a $\mathfrak{A}-\mathfrak{A}$ bimodule structure and closeness of a subspace in $H_k(P)$ under multiplication by $U$ to both sides is equivalent to its $\mathfrak{A}-\mathfrak{A}$ bimodule property. Moreover, by $A$ denote the von Neumann subalgebra in $M_k$ generated by $U$, i.e. the closure of $\mathfrak{A}$ relative to the weak topology in $M_k$. Then for a closed subspace in $L^2(M_k)$, its closedness even under multiplication by $U$ to both sides is equivalent to its $A-A$ bimodule property.

2) For $x \in P_{n,k}$, $x_{p,q}$ denotes the element in $P_{n+p+q,k}$ given by the following tangle. Here $p,q$ are the numbers of ``cups":
\begin{figure}[hb]
\centering
\includegraphics[width=0.43\textwidth]{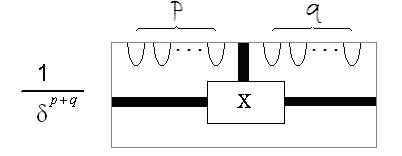}
\end{figure}
\\
In case $x \in P_{0,k}$, being dependent on $p+q=r$ only, $x_{p,q}$ can be also denoted by $x_r$.
Put $V_0=P_{0,k}$ and then for $n=1,2,\cdots$, by $V_n$ denote the orthogonal complement of $\{x_{1,0},x_{0,1}\lvert\ x \in P_{n-1,k}\}$ in $P_{n,k}$.
\end{remark}
\begin{lemma}
In case $n\geq 1$, the condition $y\in V_n$ for $y\in P_{n,k}$ is equivalent to the following pictorial equality:
\begin{figure}[h]
\centering
\includegraphics[width=0.45\textwidth]{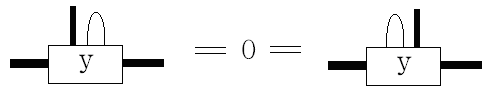}
\end{figure}
\end{lemma}
\begin{proof} By the definition of the inner product, for any $x\in P_{n-1,k}$, $\langle y,x_{0,1}\rangle$ is equal to the following:
\begin{figure}[ht]
\centering
\includegraphics[width=0.25\textwidth]{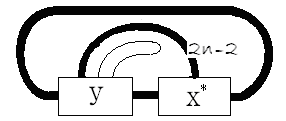}
\end{figure}
\\
Recall the definition of the inner product again. Then, denoting the element in $P_{n-1,k}$ in the left equality in the statement by $y'$, it is clear that the equality means the following equality:
$$\langle y,x_{0,1}\rangle=\langle y',x\rangle$$
By nondegeneracy of the inner product, $\langle y',x\rangle=0$ for all $x\in P_{n-1,k}$ implies $y'=0$. Analogous argument can be applied for the right equality.                   
\end{proof}
\begin{corollary} For any $x,y\in V_0$, we have $\langle x_r,y_r'\rangle=\delta_{r,r'}\langle x,y\rangle$;
In case $m+n\geq 1$, we have $\langle u_{p,q},v_{p',q'}\rangle=\delta_{p,p'}\delta_{q,q'}\langle u,v\rangle$ for any $u\in V_m$, $v\in V_n$. Here $\delta$ denotes Kronecker's delta.
\end{corollary}
\begin{proof} It is easy to see that the first statement is true. We consider the second one only.

Suppose either $p\neq p'$ or $q\neq q'$. Then it is sufficient to study only the case in which $u_{p,q}$ and $v_{p',q'}$ belong to the same $P_{l,k}$. For $\langle u_{p,q},v_{p',q'}\rangle$, picturing again the tangle gives at least one of $u$ and $v$``capping" like the previous lemma, so it vanishes.

Also, when $p=p'$ and $q=q'$ at the same time, it is sufficient to see only the case in which $u$ and $v$ belong to the same $P_{l,k}$. But now, let us regard that the $p+q$ closed bands arising in the tangle for $\langle u_{p,q},v_{p',q'}\rangle$ can be converted into $\delta^{2(p+q)}$.
Then it is clear that the inner product is conserved.
\end{proof}

Now, for $B \subset H_k(P)$, by $\mathfrak{A} B \mathfrak{A}$ denote the submodule generated by $B$.
\begin{lemma} If $n\geq 1$, then for any $u\in V_n,\ \langle u,u\rangle=1$, the set
$$\{u_{p,q}\lvert\ p,q=0,1,2,\cdots \}$$
forms an orthonormal basis for $\mathfrak{A} u \mathfrak{A}$. Therefore we can identify the completion of $\mathfrak{A} u \mathfrak{A}$ with $l^2(\mathbb{N})\otimes l^2(\mathbb{N})$ by the obvious unitary equivalence.

Moreover under this identification the left and right multiplication operators by $U$ are represented respectively as follows:
$$ (\delta (S+S^*)+id)\otimes id,\ \ id \otimes (\delta (S+S^*)+id) $$
Here $S$ denotes unilateral shift on $l^2(\mathbb{N})$, i.e.
$$ S(a_1,a_2,a_3,\cdots)=(0,a_1,a_2,\cdots),\ (a_1,a_2,a_3,\cdots) \in l^2(\mathbb{N}). $$
\end{lemma}
\begin{proof} The orthonormality of $\{u_{p,q}\lvert\ p,q=0,1,2,\cdots \}$ is an immediate consequence of the corollary 3.3.

From an easy computation, the following expansion formula is obtained:
$$ U(u_{p,q})=\left\{ \begin{array}{ll}
\delta u_{1,q}+u_{0,q}\ , &p=0\\
\delta u_{p+1,q}+u_{p,q}+ \delta u_{p-1,q}\ , &p\geq 1
\end{array}\right. $$
Here $U(u)$ denotes the action of $U$ to $u$ as a left multiplication operator. It is not difficult to see that an analogous expansion formula for the right multiplication by $U$ can be obtained symmetrically.

From above formula by putting $u_{0,0}=u$ and using an inductive argument, we can see that $u_{p,q} \in \mathfrak{A} u \mathfrak{A}$ for all $p,q=0,1,2,\cdots$. Therefore
$$Span\{u_{p,q}\lvert\ p,q=0,1,2,\cdots \} \subset \mathfrak{A} u \mathfrak{A}$$
and since clearly $Span\{u_{p,q}\lvert\ p,q=0,1,2,\cdots \}$ is closed under multiplications by $U$ to both sides, the first statement is proved.

Now put $\theta_n=(\overbrace{0,\cdots,0,1}^n,0,\cdots)$. Then $\{\theta_n\lvert\ n=0,1,2,\cdots \}$ gives an orthonormal basis for $l^2(\mathbb{N})$.
It is clear that the correspondence $u_{p,q}\mapsto \theta_p\otimes \theta_q$
gives a unitary transformation from the completion of $\mathfrak{A} u \mathfrak{A}=Span\{u_{p,q}\lvert\ p,q=0,1,2,\cdots \}$ onto $l^2(\mathbb{N})\otimes l^2(\mathbb{N})$.

Let us convert above formula on $U(u_{p,q})$ coordinately. For simplicity, by the same letter $U$ denote the representation of the operator $U$ onto $l^2(\mathbb{N})\otimes l^2(\mathbb{N})$. Then the left action of $U$ for $(a_{p,q})_{p,q=0,1,\cdots} \in l^2(\mathbb{N})\otimes l^2(\mathbb{N})$ is calculated as follows:
\begin{align*}
&U\left(\sum_{p,q=0}^\infty a_{p,q}\theta_p\otimes \theta_q\right)=U\left(\sum_{p,q=0}^\infty a_{p,q}u_{p,q}\right)\\
&=\sum_{p,q=0}^\infty \big( a_{0,q}(\delta u_{1,q}+u_{0,q})+a_{1,q}(\delta u_{2,q}+u_{1,q}+\delta u_{0,q})+ \cdots\\
& \qquad \qquad \qquad \qquad \cdots+ a_{n,q}(\delta u_{n+1,q}+u_{n,q}+\delta u_{n-1,q})+ \cdots \big)\\
&=\sum_{q=0}^\infty \sum_{p=0}^\infty \big(\delta (a_{p+1,q}+a_{p-1,q})+a_{p,q}\big) u_{p,q}\\
&=\sum_{q=0}^\infty \left( \sum_{p=0}^\infty \big(\delta (a_{p+1,q}+a_{p-1,q})+a_{p,q}\big)\theta_p\right)\otimes  \theta_q\\
&=\left[ \big(\delta (S+S^*)+id\big)\otimes id \right] \left(\sum_{p,q=0}^\infty a_{p,q}\theta_p \otimes \theta_q \right)
\end{align*} 
In the calculation, $a_{-1,q}=0$ was supposed.

For the right multiplication by $U$, an analogous calculation can be used.
\end{proof}
By corollary 2.3, for different $n\geq 1$, $\mathfrak{A} V_n \mathfrak{A}$ are orthogonal to each other. Therefore, as a $\mathfrak{A} - \mathfrak{A}$ subbimodule of $H_k(P)$, the orthogonal direct sum $\bigoplus_{n\geq 1} \mathfrak{A} V_n \mathfrak{A}$ makes sense and its closure is clearly a $A-A$ bimodule, which will be denoted by $\mathfrak{R}$.
\begin{corollary}  As a Hilbert space, $\mathfrak{R}$ is unitarily isomorphic to
$$ l^2(\mathbb{N})\otimes l^2(\mathbb{N})\otimes l^2(\mathbb{N})$$
and, under this isometry, the left and right multiplication operators by $U$ are represented respectively as follows:
$$ id\otimes (\delta (S+S^*)+id)\otimes id,\ \ id \otimes id \otimes (\delta (S+S^*)+id) $$
\end{corollary}
\begin{proof} Since $V_n$ are all finite dimensional, taking orthonormal bases from each of them and arranging them in order gives an orthonormal basis $\{u_m\}_{m\in \mathbb{N}}$ for
$$ V_1\oplus V_2\oplus \cdots \oplus V_n\oplus \cdots $$
By corollary 3.3, for different $m$, $\mathfrak{A} u_m \mathfrak{A}$ are orthogonal to each other, so that one obtains again
$$\bigoplus_{n\geq 1} \mathfrak{A} V_n \mathfrak{A}=\bigoplus_{m\geq 1} \mathfrak{A} u_m \mathfrak{A}$$
the decomposition into orthogonal direct sum.

Furthermore since $\mathfrak{A} u_m \mathfrak{A}=l^2(\mathbb{N})\otimes l^2(\mathbb{N})$ is known by the previous lemma, we conclude:
$$ \mathfrak{R}\cong \bigoplus_{n\geq 1} l^2(\mathbb{N})\otimes l^2(\mathbb{N})\cong l^2(\mathbb{N})\otimes l^2(\mathbb{N})\otimes l^2(\mathbb{N})$$
Here $\cong $ means a unitary isomorphism. It is clear that the representation formulas do not need to explain anymore.
\end{proof}
In the followings, $U\xi $ and $\xi U$ denote action of the multiplication operator by $U$   for $\xi \in L^2(M_k)$ from both sides respectively.
\begin{lemma} For $\xi \in \mathfrak{R}$, the equality $U\xi =\xi U$ implies that $\xi =0$.
\end{lemma}
\begin{proof} From corollary 3.5 we may regard $\xi $ as an element of $l^2(\mathbb{N})\otimes l^2(\mathbb{N})\otimes l^2(\mathbb{N})$, moreover,
$$\xi =(\xi_ 1,\xi _2,\xi_3,\cdots),\quad \xi_n \in l^2(\mathbb{N})\otimes l^2(\mathbb{N}).$$
In this representation the given condition $U\xi =\xi U$ implies
$$\left((\delta (S+S^*)+id)\otimes id \right)\xi_n =\left(id\otimes (\delta (S+S^*)+id) \right)\xi_n $$ for all $n=1,2,3,\cdots$.

Now it seems to need to a little mention of tensor products between Hilbert spaces.
For given Hilbert spaces $\mathcal{H}$ and $\mathcal{K}$, let $HS(\mathcal{H},\mathcal{K})$  be the Hilbert space of Hilbert-Schmidt operators from $\mathcal{H}$ to $\mathcal{K}$.
With the conjugate Hilbert space of $\mathcal{H}$ denoted by $\mathcal{H}_C$, the following correspondence gives an unitary isomorphism between $\mathcal{H}\otimes \mathcal{K}$ and 
$HS(\mathcal{H},\mathcal{K})$:
\begin{align*}
&\qquad \qquad \mathcal{H}\otimes \mathcal{K}\ni \mu \otimes \nu \mapsto T_{\mu \otimes \nu }\in \mathcal{B}(\mathcal{H},\mathcal{K}),\\
&\qquad \qquad \qquad \qquad \qquad \qquad T_{\mu \otimes \nu }\eta :=\langle \mu ,\eta \rangle \nu,\quad \mu ,\eta \in \mathcal{H},\ \nu \in \mathcal{K}
\end{align*}
It should be emphasized that for $Q\in \mathcal{L}(\mathcal{H})$, $R\in \mathcal{L}(\mathcal{K})$ and $\zeta \in \mathcal{H}\otimes \mathcal{K}$, the Hilbert-Schmidt operator corresponding to the element $(Q\otimes R)\zeta \in \mathcal{H}\otimes \mathcal{K}$  is $RT_\zeta Q^*$.

To proceed, suppose for $\zeta \in l^2(\mathbb{N})\otimes l^2(\mathbb{N}), \zeta \neq 0$, that the following equality is true:
$$\big((\delta (S+S^*)+id)\otimes id\big)\zeta =\big(id\otimes (\delta (S+S^*)+id)\big)\zeta $$
Then by above mentioned, this equality is equivalent to $T_\zeta (S+S^*)=(S+S^*)T_\zeta $.
Here $T_\zeta $, as $T_\zeta :l^2(\mathbb{N})\rightarrow l^2(\mathbb{N})$, denotes the corresponding Hilbert-Schmidt operator as well.
What should be emphasized is that $T_\zeta $ is a compact operator commutative with $S+S^*$. Therefore the compactness guarantees the existence of a nonzero eigenvalue of $T_\zeta $ with finite multiplicity.

We denote the corresponding finite dimensional eigenspace by $V$. Due to above mentioned commutativity, $V$ is an invariant subspace even for $S+S^*$. Therefore by focusing on its restriction onto $V$, we can see that $S+S^*$ has an eigenvalue. 
On the other hand, $S+S^*$ is a typical Toeplitz operator. Since it is well known that self-adjoint Toeplitz operators have no eigenvalue(e.g., see [3]), we have $\zeta =0$.

Again, since $\xi _n=0$ for every $n=0,1,2,\cdots$, we have $\xi =0$.
\end{proof}
Consider the von Neumann subalgebra in $M_k$ generated by $U$ and $P_{0,k}$.
This is both a weak closure of $\mathfrak{A}P_{0,k}\mathfrak{A}=\mathfrak{A}P_{0,k}$ in $M_k$ and a submodule generated by $P_{0,k}$ in $M_k$ as a $A-A$ bimodule. Denote it by $AP_{0,k}$.
\begin{lemma} We have $\{ U\}'\cap M_k=AP_{0,k}$.
\end{lemma}
\begin{proof} From $\mathfrak{A}P_{0,k}\subset \{ U\}'\cap M_k$, it is clear that $AP_{0,k}\subset \{ U\}'\cap M_k$.

Before checking the opposite inclusion, let us see that there is decomposition into orthogonal direct sum:
$$H_k(P)=\bigoplus_{n=0}^\infty \mathfrak{A} V_n \mathfrak{A}=\mathfrak{A}P_{0,k}\oplus \bigoplus_{n\geq 1}\mathfrak{A} V_n \mathfrak{A}$$
                      
It is sufficient to verify for every $n\geq 0$ the following inclusion:
$$P_{n,k}\subset \mathfrak{A}P_{0,k}\oplus \bigoplus_{n\geq 1}\mathfrak{A}V_n\mathfrak{A}$$

Let us go inductively. It is trivial for $n=0$.

For the sake of convenience, denote by $\omega _p$ the operation of adding $p$ cups to the tangles denoting elements of $P_{n,k}$ on the left or right side and then multiplying $\delta ^{-p}$ upon our convention of picturing. In other words, for $\omega _p$ the following condition is assumed:
$$x_{p,q}=\omega _p x\omega _q \text{ \ for } x\in P_{n,k}$$
Then from corollary 3.3 and lemma 3.4, we have the following orthogonal decompositions:
$$\mathfrak{A}P_{0,k}=\bigoplus_{n\geq 0}P_{0,k}\omega _n,\quad \mathfrak{A}V_m \mathfrak{A}=\bigoplus_{p,q\geq 0}\omega _p V_m \omega _q,\ m\geq 1$$  
Therefore the assumption $P_{n,k}\subset \mathfrak{A}P_{0,k}\oplus \bigoplus_{n\geq 1}\mathfrak{A}V_n\mathfrak{A}$ for $n$ is equivalent to the following decomposition:
\begin{align*}
& P_{n,k}=\omega_n P_{0,k}\oplus\left(\bigoplus_{p+q=n-1}\omega_p V_1\omega _q\right)\oplus\left(\bigoplus_{p+q=n-2}\omega _p V_2 \omega _q\right)\oplus \cdots \\
&\qquad \qquad \qquad \qquad \qquad \qquad \qquad \qquad \qquad \cdots \oplus \big(\omega _1 V_{n-1}\oplus V_{n-1}\omega _1 \big)\oplus V_n 
\end{align*}
It follows from the definition of $V_m$ that
$$P_{1,k}=\omega_1 P_{0,k}\oplus V_1;\quad P_{n+1,k}=\omega_1 P_{n,k}\oplus P_{n,k}\omega_1\oplus V_{n+1},\ n\geq 1.$$
Moreover, in view of the obvious equality $\omega_p\omega_q=\omega_{p+q}$, we have
$$P_{n+1,k}\subset \mathfrak{A}P_{0,k}\oplus \bigoplus_{m\geq 1}\mathfrak{A}V_m\mathfrak{A}.$$
Therefore the following orthogonal decomposition into $\mathfrak{A}-\mathfrak{A}$ modules is given.
$$H_k(P)=\mathfrak{A}P_{0,k}\oplus \bigoplus_{m\geq 1}\mathfrak{A}V_m\mathfrak{A}$$

Furthermore an orthogonal decomposition $L^2(M_k)=\overline{AP_{0,k}}\oplus \mathfrak{R}$ into $A-A$ modules is obtained through completion. Here the bar over letters indicates closure relative to the inner product.

Since, regarding $M_k$ as $M_k\subset L^2(M_k)$, the weak operator topology of $M_k$ is weaker than the norm topology of $L^2(M_k)$, we obtain $\overline{\bigoplus_{m\geq 1}\mathfrak{A}V_m\mathfrak{A}}^{wo}\subset\mathfrak{R}$.
Again, in view of the orthogonal decomposition
$$M_k=AP_{0,k}\oplus \overline{\bigoplus_{m\geq1}\mathfrak{A}V_m\mathfrak{A}}^{wo}\subset\mathfrak{R},$$
the equality $\overline{AP_{0,k}}\cap M_k=AP_{0,k}$ is obtained.

On the other hand, due to lemma 3.6, if $U\xi =\xi U$ for $\xi\in L^2(M_k)$, then we have $\xi\in \overline{AP_{0,k}}$. Therefore we have $\{ U\}'\cap M_k\subset \overline{AP_{0,k}}$  and referring to above obtained equality gives the following desired inclusion:
$$\{U\}'\cap M_k \subset \overline{AP_{0,k}}\cap M_k=AP_{0,k}$$
\end{proof}
\begin{theorem} Suppose $\delta >1$. Then we have $M_0'\cap M_k=P_{0,k}\big(=P_k \big)$.
\end{theorem}
\begin{proof} Since we have clearly $P_{0,k}\subset M_0'\cap M_k$, it is sufficient to prove the opposite inclusion.
In view of the inclusion
$$M_0'\cap M_k\subset \{U,W\}'\cap M_k,$$
let us consider $\{U,W\}'\cap M_k$. Moreover this is equal to $\{W\}'\cap AP_{0,k}$ by the previous lemma, and again by the decomposition $\mathfrak{A}P_{0,k}=\bigoplus_{n\geq 0}P_{0,k}\omega _n$ in its proof, we have
$$\overline{AP_{0,k}}=\bigoplus_{n\geq 0}P_{0,k}\omega _n.$$
Here the direct sum in the latter is the orthogonal decomposition as a Hilbert space.
Therefore every element in $\overline{AP_{0,k}}$ has unique formal expansion:
$$c=\sum_{n=0}^\infty c_n\omega _n=\sum_{n=0}^\infty c_n\circ I_n,\quad c_n\in P_{0,k},\ \sum_{n=0}^\infty \langle c_n,c_n\rangle <\infty $$
According to our convention, $I_n$ is represented by the tangle in the following picture, which is clearly equal to $\omega _nI$ for $I\in P_{0,k}$.
\begin{figure}[ht]
\centering
\includegraphics[width=0.3\textwidth]{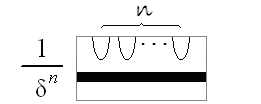}
\end{figure}
\\
Putting $\alpha _n=W\omega _n-\omega _nW$, we have $\alpha _n\in P_{n+1,k}$ and especially $\alpha _0=0$.

On the other hand, for $n\geq 1$, in view of
\begin{figure}[ht]
\centering
\includegraphics[width=0.35\textwidth]{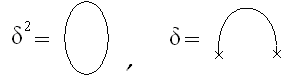}
\end{figure}
\\
through diagrammatic calculation of the inner product, we obtain
$$\langle \alpha _n,\alpha _n\rangle=2(\delta ^2-1).$$
Here $\alpha _n\neq 0$ due to $\delta >1$.
Consequently, from
\begin{align*}
W\circ I_n-I_n\circ W
&=(W\omega _n+W\omega _{n-1}+\delta I_{n-1})-(\omega _nW+\omega _{n-1}W+\delta I_{n-1})\\
&= \alpha _n+\alpha _{n-1}
\end{align*}
we obtain:
$$W\circ c-c\circ W=\sum_{n=1}^\infty c_n\circ (\alpha _n+\alpha _{n-1})=\sum_{n=1}^\infty (c_n+c_{n+1})\circ \alpha _n$$

Suppose $W\circ c=c\circ W$ for $c \in \overline{AP_{0,k}}$. Then, since above expansion is orthogonal, we have $(c_n+c_{n+1})\circ \alpha _n=0$ for all $n\geq 1$.

Now let us recall that the quality $a\circ \alpha=0$ always implies either $a=0$ or $\alpha =0$, because it is easy to see that
$\langle a\circ \alpha ,a\circ \alpha \rangle =\langle a,a\rangle \langle \alpha ,\alpha \rangle $ for any $a \in P_{0,k}$ and $\alpha \in P_{n,0}$.

Since $\alpha \neq 0$ is known for all $n\geq 1$, by proceeding above consideration, we obtain $c_n+c_{n+1}=0$, i.e. $c_n=-c_{n+1}$.
Therefore by recalling $\sum_{n=0}^\infty \langle c_n,c_n\rangle <\infty $, we have $c_n=0$ for all $n\geq 1$ and after all we conclude $c \in P_{0,k}$.

As a result, we have $\{W\}'\cap AP_{0,k} \subset P_{0,k}$, which implies $M_0'\cap M_k \subset P_{0,k}$.
\end{proof}
\begin{remark} We can get formally the proof of above theorem from [6] (4.2-4.11) by changing the meaning of lines with black bands and, in the places where the factor $\sqrt{\delta }$ appears, replacing it by $\delta $. Besides, in our discussion above, even by changing the roles $U$ and $W$ with each other, a similar proof will be given as well.
\end{remark}
\begin{corollary} Under the same assumption, i.e. $\delta >1$, $M_k$ is a $II_1$ factor.
\end{corollary}
\begin{proof} Let $\mathcal{Z}(M_k)$ denote the center of $M_k$. It is sufficient to prove that $\mathcal{Z}(M_k)=\mathbb{C}$ is true.

First of all, by the previous theorem we have $\mathcal{Z}(M_k)\subset P_{0,k}$.
Now let us examine the following tangle $\Lambda \in H_k(P)\big(\subset M_k \big)$:
\begin{figure}[ht]
\centering
\includegraphics[width=0.13\textwidth]{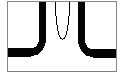}
\end{figure}
\\
What is to be careful is the explanation of the lines running upward:
At a look, it is clear that the thick lines should have the same meaning no matter what direction they run toward. We assume that all the thick lines mean bundles of $k$ strings.
But the understanding of the thin one running upward is different according to parity of $k$.
While in case $k=2k'$, it means a black band as before, it means a string in case $k=2k'+1$.
On the whole, we have $\Lambda \in P_{k'+1,k}$.

To proceed with, for any fixed $a \in P_{0,k}$ orthogonal to the unit $I$, suppose its commutativity with $\Lambda $.
Notice that the orthogonality to $I$ is equivalent to the equality $Tr(a)=Tr(a^*)=0$.
Then the equality
$$\langle a\circ\Lambda ,a\circ\Lambda \rangle =\langle \Lambda \circ a,a\circ\Lambda \rangle$$implied from $a\circ \Lambda =\Lambda \circ a$, means the equivalence of the following two tangles:
\begin{figure}[ht]
\centering
\includegraphics[width=0.65\textwidth]{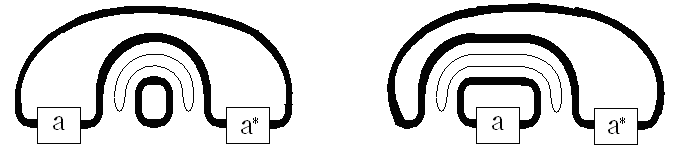}
\end{figure}
\\
Since this equivalence means $\delta ^k Tr(a^*a)=Tr(a)Tr(a^*)$, we have $Tr(a^*a)=0$.
This, in view of the nondegeneracy of the trace, implies $a=0$.
Therefore we have
$$\{\Lambda \}'\cap P_{0,k}\subset \mathbb{C}\big(=\mathbb{C}I\big).$$

Since it is clear that $\mathbb{C}\subset \{\Lambda \}'\cap P_{0,k}$, after all, we conclude $\mathcal{Z}(M_k)=\mathbb{C}$.
\end{proof}

\section{the invariant of subfactors and the planar algebras}
Let us study the relation between the inclusions of $II_1$ factors
$$M_0 \subset M_1 \subset M_2 \subset \cdots$$
and the basic construction for them.
\begin{definition} 1) According to our convention, for every $k=0,1,2,\cdots$, as an element in $P_{0,k+1}\big(\subset P_{k+1}\big)$, the ordinary Jones projection tangle $E_k$ is given by left tangle in the figure bellow(Fig.7). 
We denote $\delta^{-1} E_k$ by $e_k$. It represents clearly an idempotent in $M_{k+1}$.

2) Let us fix $l=1,2,3,\cdots$. Then for every $k=0,1,2,\cdots$ a mapping from $H_{k+l}(P)$ to $H_k(P)$ is given componentwise by the ordinary conditional expectation tangles from $P_{2n+k+l}$ to $P_{2n+k}$, $n=0,1,2,\cdots$, every of which as a mapping from $P_{n,k+l}$ to $P_{n,k}$, according to our convention, would be seen as right one in the following figure.
\begin{figure}[ht]
\centering
\includegraphics[width=0.55\textwidth]{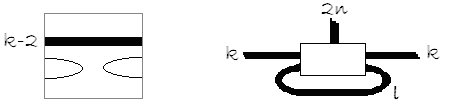}
\caption{}
\end{figure}
\\
We denote this mapping $H_{k+l}(P)\rightarrow H_{k}(P)$ by $\mathcal{E}_k^{k+l}$.
\end{definition}
\begin{remark} It is not difficult to see that the mapping $\delta ^{-l}\mathcal{E}_k^{k+l}$ agrees with the restriction of the conditional expectation $E_{M_k}:M_{k+l}\rightarrow M_k$  onto $H_{k+l}(P)$.
\end{remark}
\begin{lemma} For every $e_k,\ k=1,2,3,\cdots$, we have $E_{M_k}(e_k)=\delta^{-2}$ and moreover the following relations are satisfied:
$$tr(xe_k)=\delta^{-2}tr(x),\ e_kxe_k=E_{k-1}(x)e_k,\ x\in M_k$$
\end{lemma}
\begin{proof} The first equality is clear from above definition and the remark.

Let us examine the others. By weak continuity of the trace and the conditional expectations, there is no loss of generality to see it only in case of $x\in H_k(P)$.
And moreover, this case is reduced to the case where $x\in P_{n,k}$.
Furthermore to see the equality $tr(xe_k)=\delta^{-2}tr(x)$, in view of remark 2.9, 2) it is sufficient with only case $x\in P_{0,k}$. The rest is trivial from diagrammatic consideration.

For the last one we come to check equality of the following two tangles, the fact which is obvious:
\begin{figure}[ht]
\centering
\includegraphics[width=0.7\textwidth]{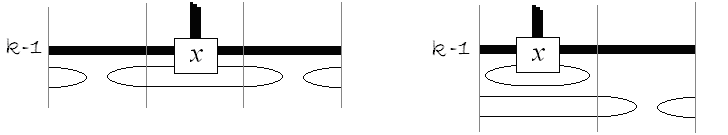}
\end{figure}
\\
\end{proof}
We shall restrict our attention to $M_0 \subset M_1 \subset M_2$.
Let us denote by $\mathfrak{M}_2$ the von Neumann subalgebra in $M_2$ generated by $M_1$ and $e_1$ and focus our attention on it.
\begin{corollary} For every $x\in \mathfrak{M}_2$, there is unique $m\in M_1$ such that $xe_1=me_1$.
\end{corollary}
\begin{proof} Suppose that for $x\in \mathfrak{M}_2$, there is $m\in M_1$ such that $xe_1=me_1$. Then apply to both sides of the equality the conditional expectation $E_{M_1}$.

By lemma 4.3, $E_{M_k}(e_k)=\delta^{-2}$, we get $m=\delta^{-2} E_{M_1}(xe_1)$.
Again due to weak continuity of the conditional expectation, the only thing we have to do is to verify the existence of $m\in M_1$ satisfying $xe_1=me_1$ for the elements $x$ in a subspace dense in $\mathfrak{M}_2$.

Now it is sufficient to point out for $x=a+\sum a_ie_1b_i,\ a,a_i,b_i\in M_1$, the following equality which is obtained from lemma 4.3 immediately:
$$xe_1=\big(a+\sum a_i E_{M_0}(b_i)\big)e_1$$ 
\end{proof}
\begin{lemma} $\mathfrak{M}_2$ is a $II_1$ factor.
\end{lemma}
\begin{proof} Suppose that $\bar{a}\in \mathfrak{M}_2$ is commutative with all the elements in $\mathfrak{M}_2$, i.e., that $\bar{a}\in \mathcal{Z}(\mathfrak{M}_2)$. Anyway, since $\bar{a}\in M_0'\cap M_2$, by theorem 2.8, $\bar{a}\in P_{0,2}$.

Consider the element $\Lambda \in P_{1,2}$ described by the following tangle (all the lines imply strings):
\begin{figure}[ht]
\centering
\includegraphics[width=0.17\textwidth]{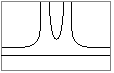}
\end{figure}
\\
In the sense of the embedding $\mathcal{I}_1^2: H_1(P)\rightarrow H_2(P)$, we have $\Lambda \in M_1 \subset \mathfrak{M}_2$.
And $\overline{P_{0,1}}\subset P_{0,2}$, the range of $P_{0,1}$ under the inclusion by the following tangle, is clearly commutative with $\Lambda $.
\begin{figure}[ht]
\centering
\includegraphics[width=0.19\textwidth]{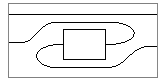}
\end{figure}
\\
Therefore consider first the case when $\bar{a}$ is orthogonal to $\overline{P_{0,1}}$.
A simple calculation shows that the orthogonality implies the following identity:
\begin{figure}[ht]
\centering
\includegraphics[width=0.23\textwidth]{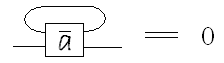}
\end{figure}
\\
In view of this identity, starting from $\langle \bar{a}\Lambda ,\bar{a}\Lambda \rangle =\langle \Lambda \bar{a},\bar{a}\Lambda \rangle$, by a diagrammatic calculation analogous to one in the proof of corollary 3.10, we get $Tr(\bar{a}^*\bar{a}=0$ from and therefore $\bar{a}=0$.
After all, we have $\bar{a}\in \overline{P_{0,1}}$. In other words, $\bar{a}$ is given by the following tangle for some $a\in P_{0,1}$:
\begin{figure}[ht]
\centering
\includegraphics[width=0.2\textwidth]{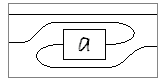}
\end{figure}
\\
Again, in view of its commutativity with $e_1$, by the same way with one in 3.10 we have $a=\delta ^{-1}Tr(a)\in \mathbb{C}$, which implies $\bar{a}\in \mathbb{C}$.
\end{proof}
\begin{lemma} $M_0 \subset M_1 \subset \mathfrak{M}_2$ is the basic construction for $M_0 \subset M_1$. In other words, there is a surjective $\ast $-isomorphism
$$\phi :\mathfrak{M}_2\rightarrow \langle M_1,e_{M_0}\rangle,\quad \phi (M_0)=M_0,\ \phi (M_1)=M_1$$
Here $\langle M_1,e_{M_0}\rangle$ is the basic construction for $M_0 \subset M_1$ and $e_{M_0}$ is the orthogonal projection $e_{M_0}:L^2(M_1)\rightarrow L^2(M_0)$.
And we have the Jones index $[M_1:M_0]=\delta ^2$.
\end{lemma}
\begin{proof} Consider the standard form of $\mathfrak{M}_2$, i.e., its action by the left multiplication on $L^2(\mathfrak{M}_2)$.
Since $\mathfrak{M}_2e_1=M_2e_2$ by corollary 4.4, moreover, we have $[\mathfrak{M}_2e_1]=[M_2e_1]$. Here $[\ ]$ denotes the closure in $L^2(\mathfrak{M}_2)$.

It is clear that $\mathfrak{M}_2[M_1e_1]\subset [M_1e_1]$.
Therefore, if we denote the projection onto $[M_1e_1]$ by $p$, then $p\in \mathfrak{M}_2$.
Since $\mathfrak{M}_2$ is a factor by lemma 4.5,
$$\phi ':\mathfrak{M}_2\rightarrow (\mathfrak{M}_2)_p,\ \phi '(x):=x_p,$$
the induced mapping is a surjective $\ast $-isomorphism and gives on $[M_1e_1]$ a faithful $\ast $-representation of $M_0 \subset M_1 \subset \mathfrak{M}_2$.
 
Notice that the correspondence $M_1\ni x\mapsto \delta xe_1\in M_1e_1$ determines a unitary transform $U:L^2(M_1)\rightarrow [M_1e_1]$. Indeed the equality bellow, which follows from lemma 4.3, implies that the mapping is an isometry.
$$tr((xe_1)^\ast xe_1)=tr(e_1x^\ast xe_1)=tr(E_{M_0}(x^\ast x)e_1)=\delta ^{-2}tr(x^*x)$$ 

Now consider the behavior of $(M_0 \subset M_1 \subset \mathfrak{M}_2)_p$ under the spatial isomorphism by $U$,
$$\phi '':\mathcal{B}\big([M_1e_1]\big)\rightarrow \mathcal{B}\big(L^2(M_1)\big),\ \phi ''(x):=U^*xU.$$
Notice that the action of $x_p\in (\mathfrak{M}_2)_p$ on $[M_1e_1]$ is given by $x_p(ye_)=xye_1,\ y\in M_1$.
Then first of all, since for $x\in M_1$, 
$$x_pU(y)=x_p(\delta ye_1)=\delta xye_1=U(xy),\ y\in M_1$$
we have $(\phi ''(x_p))y=xy$.
Therefore, under the identification of the elements of $M_1$ with the corresponding left multiplications on $L^2(M_1)$(i.e. the standard representation), we have $\phi ''\left((M_1)_p\right)=M_1$ and hence $(\phi ''\circ \phi ')\arrowvert _{M_1}=id_{M_1}$. On the other hand, for $e_1$, from
$${e_1}_p(U(y))={e_1}_p(\delta ye_1)=\delta e_1ye_1=\delta E_{M_0}(y)e_1=U(E_{M_0}(y)),\ y\in M_1$$   
it follows that $\phi ''({e_1}_p)\arrowvert _{M_1}=E_{M_0}$, and furthermore we get $(\phi ''\circ \phi ')(e_1)=e_{M_0}$.
After all it is proved that the mapping $\phi =\phi ''\circ \phi '$ is the desired $\ast $-isomorphism. Therefore $M_0 \subset M_1 \subset \mathfrak{M}_2$ is the basic construction and $e_1$ is its Jones projection.
Now it is also clear that we have $[M_1:M_0]={tr(e_1)}^{-1}=\delta ^2$.               
\end{proof}
\begin{corollary} We have $M_2=\mathfrak{M}_2$. Therefore $M_0 \subset M_1 \subset M_2$ is a basic construction.
\end{corollary}
\begin{proof} It is easy to see that, even for $\subset M_1 \subset M_2 \subset M_3$, we can repeat the argument like one in the previous lemma. Therefore we get $[M_2:M_1]=\delta ^2$.
Since $[\mathfrak{M}_2:M_1]=[M_1:M_0]=\delta ^2$ by the previous lemma, we see that $[\mathfrak{M}_2:M_2]=1$ and therefore $\mathfrak{M}_2=M_2$.
\end{proof}
\begin{theorem} $M_0 \subset M_1 \subset M_2 \subset \cdots$ is a tower of basic construction and $e_k,\ k=1,2,\cdots$, are its Jones projections.

Moreover we have ${M_0}'\cap M_k=P_k$. Besides, for every $l$ the restriction of the conditional expectation $E_{M_k}:M_{k+l}\rightarrow M_k$ to ${M_0}'\cap M_{k+l}$
$$E_{{M_0}'\cap M_k}:P_{k+l}\rightarrow P_k$$
is equal to the corresponding conditional expectation tangle (more exactly, the mapping described by it) multiplied by $\delta ^{-1}$.
\end{theorem}
\begin{proof} This is nothing but repeating the theorem 3.8 and 4.1-4.8 inductively.
\end{proof}
\begin{corollary} For every subfactor planar algebra, the standard invariant of the subfactor $M_0 \subset M_1$ above constructed, i.e. its planar algebra is identical with $P$.
\end{corollary}
\begin{proof} By an argument analogous to lemma 4.5, it is not difficult to see that ${M_1}'\cap M_{k+l}$ is identical with the range of $M_k$ by the following tangle:
\begin{figure}[ht]
\centering
\includegraphics[width=0.22\textwidth]{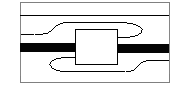}
\end{figure}
\\
It is also clear that on ${M_0}'\cap M_k$, $\delta E_{{M_1}'\cap M_k}$ is represented by the ordinary left conditional expectation tangle of $P$, which can be seen in our convention as follows:
\begin{figure}[ht]
\centering
\includegraphics[width=0.22\textwidth]{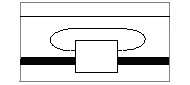}
\end{figure}
\\
From these facts combined with theorem 3.9, we see that all the basic ingredients ([5] or, for more detailed, [7]) which determine a subfactor planar algebra over the lattice of relative commutants accompanied with, are completely identical with those of $P$.
\end{proof}

\section{Comments}
Our approach provides further generalization.
Let us start with any given subfator planar algebra $P=(P_{n})_{n=0_{\pm},1,2,\cdots}$.
For every $r=1,2,3,\cdots$, let us introduce the following sequence of direct sums (of finite dimensional vector spaces):
$$H_k^{(r)}(P)=\bigoplus _{n=0}^\infty P_{n,k}^{(r)},\ k=0,1,2,\cdots$$
Here by $P_{n,k}^{(r)}$ we denote $P_{2nr+k}$.

The discussion in the previous sections is corresponding to case of $r=1$. Furthermore it seems routine to verify the analogous statements for the situation here. Replacing Fig.1, the manner of describing elements in $P_{n,k}$ by the figure bellow, makes every detail clear.
Only thing we have to pay attention is to be careful with the meaning of the thick lines running upward in Fig.2 and therefore with the factors which might appear in connection with the modulus $\delta $.
\begin{figure}[ht]
\centering
\includegraphics[width=0.85\textwidth]{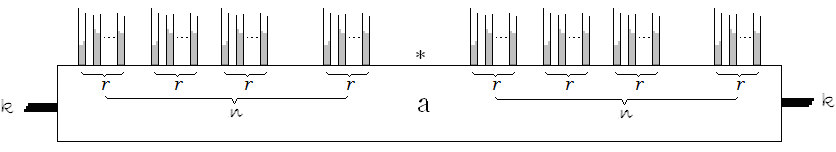}
\end{figure}

For every $r=1,2,3,\cdots$,
$$H_0^{(r)}(P)\subset H_1^{(r)}(P)\subset H_2^{(r)}(P)\subset \cdots \subset H_k^{(r)}(P)\subset \cdots $$
is a sequence of Hilbert algebras.
Moreover from these sequences we have the following a family of countably many towers of subfactors:
$$M_0^{(r)}\subset M_1^{(r)}\subset M_2^{(r)}\subset \cdots ,\ r=1,2,3,\cdots $$
It would be also routine to verify that all subfactors in every tower have the same standard invariant, i.e. the lattices of their relative commutants generate the given subfactor planar algebra.

At present it is uncertain for different $r=1,2,3,\cdots $, whether there is an equivalence between the corresponding towers, in a reasonable meaning, or not.

In [2] the isomorphism class of the factors constructed in [1] was identified in case of finite-depth(also [9], [10]). The same kind of identification problem for our construction is remained open as well.

\end{document}